\documentclass[preprint,12pt,3p]{elsarticle}

\usepackage{amssymb}
\usepackage{amsthm}
\usepackage{amsmath}
\usepackage{mathtools}
\usepackage{enumerate} 
\usepackage[hidelinks,colorlinks=false]{hyperref}
\newtheorem{theorem}{Theorem}
\newtheorem{lemma}[theorem]{Lemma}

\begin{document}

\begin{frontmatter}

\title{On the chromatic number of\\ generalized Kneser hypergraphs}

\author[label1, label2]{Hamid Reza Daneshpajouh}%\corref{cor1}\fnref{label3}}
\address[label1]{School of Mathematics, Institute for Research in Fundamental Sciences (IPM),
Tehran, Iran, P.O. Box 19395-5746}
\address[label2]{Moscow Institute of Physics and Technology, Institutsky lane 9, Dolgoprudny, Moscow region, 141700}
%\address[label2]{Address Two\fnref{label4}}

%\cortext[cor1]{I am corresponding author}
%\fntext[label3]{I also want to inform about\ldots}
%\fntext[label4]{Small city}

\ead{hr.daneshpajouh@ipm.ir, hr.daneshpajouh@phystech.edu}

%\author[label5]{Author Two}
%\address[label5]{Some University}
%\ead{author.two@mail.com}

%\author[label1,label5]{Author Three}
%\ead{author.three@mail.com}
\begin{abstract}
The generalized Kneser hypergraph $KG^{r}(n,k,s)$ is the hypergraph whose vertices are all the $k$-subsets of $\{1,\ldots ,n\}$, and edges are $r$-tuples of distinct vertices such that any pair of them has at most $s$ elements in their intersection. In this note, we show that for each non-negative integers $k, n, r, s$ satisfying $n \geq r(k-1)+1$, $k > s\geq 0$, and $r\geq 2$, we have
$$\chi ({KG}^{r}(n,k,s))\geq\left\lceil\frac{n-r(k-s-1)}{r-1}\right\rceil,$$
which extends the previously known result by Alon--Frankl--Lov\'{a}sz.
\end{abstract}

\end{frontmatter}
\section{Introduction}
Throughout this note, the set of $\{1,\ldots, n\}$ will be denoted by $[n]$. Recall, a hypergraph $\mathcal{H}$ is a pair $\mathcal{H} = (V, E)$ where $V$ is a finite set of elements called vertices, and $E$ is a set of non-empty subsets of $V$ called edges. The vertex set and the edge set of a hypergraph $\mathcal{H}$ are often denoted by $V(\mathcal{H})$ and $E(\mathcal{H})$, respectively. An $m$-coloring of a hypergraph $\mathcal{H}$ is a map $c: V(\mathcal{H})\to\{1,\ldots, m\}$. Moreover, $c$ is called a proper $m$-coloring if it creates no monochromatic edge, i.e., $|c(e)|\geq 2$ for all $e\in\mathcal{H}$. We say a hypergraph is $m$-colorable if it admits a proper $m$-coloring. The chromatic number of a hypergraph $\mathcal{H}$, denoted by $\chi (\mathcal{H})$, is the minimum $m$ such that $\mathcal{H}$ is $m$-colorable. Finally, the generalized Kneser hypergraph ${KG}^{r}(n,k,s)$ is a hypergraph whose vertex set and edge set are as follows
\begin{align*}
& V(KG^{r}(n,k,s))=\{X\subseteq [n]: |X|=k\}\\
& E(KG^{r}(n,k,s))=\{\{X_1, \ldots, X_r\}: X_i\in V(KG^{r}(n,k,s))\quad \&\quad |X_i\cap X_j|\leq s\,\,\text{for all}\, i\neq j\}.
\end{align*}
In $1955$, Martin Kneser conjectured that it is impossible to partition all $k$-subsets of $[n]$ into less than $n-2(k-1)$ classes such that any two distinct $k$-sets inside one class have non-empty intersection, where $n\geq 2k$. Or equivalently, in the language of graph coloring  
$$\chi ({KG}^2(n, k, 0))\geq n-2(k-1)\quad\text{for all}\,\, n\geq 2k.$$ 
After about two decades, this conjecture was confirmed by Lov\'{a}sz~\cite{Lovasz} via a tool from algebraic topology! Since then many other proofs have been found, see for instance~\cite{Daneshpajouh1, Greene, Karasev, Matousek}. 
Later, a generalization of Kneser's conjecture was raised by Erd\H{o}s~\cite{Erdos}. Erd\H{o}s conjectured that this is also impossible to partition all $k$-subsets of $[n]$ into less than $\left\lceil\frac{n-r(k-1)}{r-1}\right\rceil$ classes such that among any $r$ distinct $k$-stets inside one class at least two of them has non-empty intersection, where $n\geq rk$. Similarly, this is equivalent to saying that:
\begin{equation}
\chi ({KG}^{r}(n,k,0))\geq\left\lceil\frac{n-r(k-1)}{r-1}\right\rceil\quad\text{for all}\,\, n\geq rk.
\end{equation}

This conjecture originally confirmed by Alon--Frankl--Lov\'{a}sz~\cite{Alon} using some topological tools. We refer the reader to~\cite{Alishahi, Daneshpajouh2, Florian, Meunier1} for more recent proofs. 

In this note, we show that the following generalization of the Erd\H{o}s conjecture is also true.  
\begin{theorem}
Let $k, n, r, s$ be non-negative integers where $n \geq r(k-1)+1$, $k > s\geq 0$ and $r\geq 2$. We have
$$\chi ({KG}^{r}(n,k,s))\geq\left\lceil\frac{n-r(k-s-1)}{r-1}\right\rceil.$$
\end{theorem}
The organization of the paper is as follows. Section $2$ introduces our main tool, $\mathbb{Z}_p$-Tucker's lemma. And, Section 3 is devoted to the proof of Theorem $1$.
\section*{Acknowledgements}
I would like to thank Professors Meysam Alishahi and Roman Karasev for many fruitful discussions and helpful suggestions. I also wish to express my sincere gratitude to Professor Andrey Raigorodskii for drawing my attention to such kind of problem, and all his support.

\section{$\mathbb{Z}_p$-Tucker lemma}
The proof of our theorem is based on a combinatorial tool that is originally introduced by Ziegler~\cite{Ziegler}, called $\mathbb{Z}_p$-Tucker's lemma. There is also another version of this lemma~\cite{Meunier1}, with the same name. We need the latter version. To state the lemma, we need some notation.
Let $\mathbb{Z}_p=\{\omega, \omega^2, \ldots, \omega^p\}$ be the cyclic group of order $p$. Furthermore, assume that $0\notin\mathbb{Z}_p$.
For every $x=(x_1,\ldots ,x_n)\in {\left(\mathbb{Z}_p\cup\{0\}\right)}^{n}$, and $0\leq i\leq p$, set:
\[   
X_{i} = 
     \begin{cases}
       \{j : x_j= 0\} &\quad\text{if}\quad i=0 \\
       \{j : x_j= \omega^i\} &\quad\text{if}\quad i\geq 1.
     \end{cases}
\]
Also, for each $x=(x_1,\ldots ,x_n)\in {\left(\mathbb{Z}_p\cup\{0\}\right)}^{n}$, and $\omega^i\in\mathbb{Z}_p$, define:
$$\omega^i\cdot (x_1, \ldots, x_n) = (\omega^i\cdot x_1, \ldots, \omega^i\cdot x_n).$$
Note that, in above $\omega^i\cdot 0$ is defined as $0$. Finally, for $x, y\in {\left(\mathbb{Z}_p\cup\{0\}\right)}^{n}$ we write $x\preceq y$ if  $X_{i}\subseteq Y_{i}$ for all $1\leq i\leq p$. Now, we are in a position to state our main tool.

\begin{lemma}[$\mathbb{Z}_p$-Tucker lemma~\cite{Meunier1}]
Let $p$ be a prime number, and $n, m, \alpha$ be positive integers where $m\geq\alpha$. Moreover, assume that
\begin{align*}
  \lambda: & {\left(\mathbb{Z}_{p}\cup\{0\}\right)}^{n}\setminus\{(0,\ldots , 0)\}\longrightarrow\mathbb{Z}_p\times[m]\\
  & x\longmapsto (\lambda_{1}(x), \lambda_{2}(x)).
\end{align*}
is a map satisfying the following conditions:
\begin{itemize}
    \item $\lambda$ is a $\mathbb{Z}_p$-map, i.e., $\lambda(\omega^i\cdot x)= (\omega^i\cdot\lambda_{1}(x), \lambda_{2}(x))$ for all $x\in {\left(\mathbb{Z}_{p}\cup\{0\}\right)}^{n}\setminus\{(0,\ldots , 0)\}$ and all $\omega^{i}\in\mathbb{Z}_p$. 
    \item for all $x^{(1)}\preceq x^{(2)}$, if $\lambda_{2}(x^{(1)})= \lambda_{2}(x^{(2)})\leq\alpha$, then $\lambda_{1}(x^{(1)})= \lambda_{1}(x^{(2)})$. 
    \item for all $x^{(1)}\preceq\cdots\preceq x^{(p)}$ if $\lambda_{2}(x^{(1)})= \cdots =\lambda_{2}(x^{(p)})\geq\alpha+1$, then $\lambda_1(x^{(1)}), \ldots, \lambda_p(x^{(p)})$ are not pairwise disjoint.
\end{itemize}
Then, $\alpha + (m-\alpha)(p-1)\geq n$.
\end{lemma}

\section{Proof of Theorem 1}
The outline of the proof is as follows. The proof is divided into two parts. The result is first proved for the case that $r$ is a prime number, with the aid of $\mathbb{Z}_p$-Tucker's lemma. Then, we extend it to non-prime cases by a simple lemma which reduces a non-prime case to a prime case. For simplicity, let us define a notation.
For a finite set $X\subseteq\mathbb{Z}$ and any integer $0\leq l\leq |X|$, set
\[   
f_{l}(X) = 
     \begin{cases}
       \emptyset &\quad\text{if}\quad l=0 \\
       \text{the set of the first}\, l\,\text{elements of}\, X  &\quad\text{if}\quad l\geq 1.
     \end{cases}
\]
\textbf{Proof for the case that $r=p$ is a prime number:} 
\\
Let $c : V\left({KG}^{p}(n, k, s)\right)\to\{1, \ldots , C\}$ be a proper coloring of $KG^{p}(n, k, s)$ with $C$ colors. Put $\alpha= p(k-s-1)$, and $m= p(k-s-1)+ C$. In order to prove the theorem we will define a map
\begin{align*}
  \lambda:  & {\left(\mathbb{Z}_{p}\cup\{0\}\right)}^{n}\setminus\{(0,\ldots , 0)\}\longrightarrow\mathbb{Z}_p\times[m]\\
  & x\longmapsto (\lambda_{1}(x), \lambda_{2}(x))
\end{align*}
satisfying the required properties in the $\mathbb{Z}_p$-Tucker lemma. Let $x=(x_1,\ldots ,x_n)\in {\left(\mathbb{Z}_p\cup\{0\}\right)}^{n}\setminus\{(0,\ldots , 0)\}$.  Now, we consider two cases.
\begin{enumerate}[(i)]
     \item
If $|X_i|\leq k-s-1$ for all $i\in [p]$, then set
$$\lambda (x)= \left(\omega^{j}, \sum_{i=1}^{p}|X_i|\right),$$
where $\omega^{j}$ is the first nonzero element of $x$. 
\item
Consider the case that at least one of the $|X_1|, \ldots, |X_p|$ is at least $k-s$. First, note that, it is impossible that all members of $\{\left(|X_{i}|+|X_{0}|\right) : 1\leq i\leq p\}$ are simultaneously less than $k$. Since otherwise,
$$p(k-1)\geq\sum_{i=1}^{p}\left(|X_i|+|X_0|\right)= \underbrace{\sum_{i=0}^{p} |X_i|}_{n}+(p-1)|X_{0}|\geq n ,$$ which contradicts our assumption that $n\geq p(k-1)+1$. So, in this case, there is an $i\in [p]$ such that $|X_i|\geq k-s$ and $|X_i|+|X_0|\geq k$. To see this, it is enough to take $X_i$'s with the maximum size. Now, suppose $\omega^j$ be the first non-zero element of $x$ such that $|X_j|+|X_0|\geq k$ and $|X_j|\geq k-s$. Then, consider the following $k$-subset of $X_{j}\cup X_{0}$ 
$$F_{x}^{j}= f_t(X_j)\cup f_{(k-t)}(X_0),$$
where $t= \min\{k, |X_j|\}$.
In other words, $F_{x}^{j}$ is the set of the first $k$ elements of $X_{j}$, if $|X_{j}|\geq k$. Otherwise, $F_{x}^{j}$ is the union of $X_{j}$ and the first $k-|X_j|$ elements of $X_{0}$. Now, set
$$\lambda (x) = \left(\omega^j, p(k-s-1)+ c(F_{x}^j)\right).$$

%Then, choose a $k$-subset $F_{x}^{j}\subseteq X_{j}\cup X_{0}$ in a way that $|F_{x}^{j}\cap X_{j}|= \min\{k, |X_j|\}$, and if $p$ belongs to $F_{x}^{j}\cap X_{j}$ ($F_{x}^{j}\cap X_{0}$), then any smaller element than $p$ in $X_{j}$ ($X_0$) belongs to $F_{x}^j$. In other words, if $|X_{j}|\geq k$, then $F_{x}^{i}$ is the set of the first $k$ elements of $X_{j}$. And if $|X_{j}|\leq k-1$, then $F_{x}^{j}$ is the union of $X_{j}$ and the first $k-|X_j|$ elements of $X_{0}$. Now, set
It is easy to check that $\lambda$ is a $\mathbb{Z}_p$-map. So, to use the $\mathbb{Z}_p$-Tucker lemma, it is enough to show that $\lambda$ has the two other properties mentioned in the lemma.

\end{enumerate}

\begin{itemize}
    \item 
    Let $x\preceq y\in {(\mathbb{Z}_P\cup\{0\})}^n\setminus\{(0,\ldots , 0)\}$, and $\lambda_2(x)=\lambda_2(y)\leq\alpha$.
    In this case, we have $X_i\subseteq Y_i$ for all $i\in [p]$, and moreover $\sum_{i=1}^{p}|X_i|=\sum_{i=1}^{p}|Y_i|$. These imply that $x=y$. Thus, $\lambda_1(x)= \lambda_1(y)$.
    \item 
    For the second case, assume $x^{(1)}\preceq\cdots \preceq x^{(p)}\in {(\mathbb{Z}_P\cup\{0\})}^n\setminus\{(0,\ldots , 0)\}$ with $\lambda_2(x^{(1)})= \cdots= \lambda_2(x^{(p)})\geq\alpha+1$. By the definition of $\lambda$, for each $1\leq i\leq p$ there is an $1\leq l_i\leq p$, and $F_{x^{(i)}}^{l_i}\subseteq X_{l_i}^{(i)}\cup X_{0}^{(i)}$ such that
    $$\lambda (x^{(i)})= (\omega^{l_i}, p(k-s-1)+ c(F_{x^{(i)}}^{l_i})).$$
    First of all, $c(F_{x^{(1)}}^{l_1})= \cdots= c(F_{x^{(p)}}^{l_p})$ as $\lambda_2(x^{(1)})= \cdots= \lambda_2(x^{(p)})$. Note that, if for some $i < j$, we have $l_i\neq l_j$, then $X_{l_i}^{(i)}\cap X_{l_j}^{(j)}=\emptyset$ as $x^{(i)}\preceq x^{(j)}$ and the fact that $X_{l_i}^{(j)}\cap X_{l_j}^{(j)}=\emptyset$. This implies 
    $$|F_{x^{(i)}}^{l_i}\cap F_{x^{(j)}}^{l_j}|= |(F_{x^{(i)}}^{l_i}\cap X_0^{(i)})\cap (F_{x^{(j)}}^{l_j}\cap X_0^{(j)})|\leq |(F_{x^{(i)}}^{l_i}\cap X_0^{(i)})|\leq s.$$
    Therefore, all of $l_i$ cannot be pairwise distinct. Since otherwise, as we discussed above, any pair of distinct $k$-sets $F_{x^{(i)}}^{l_i}$, and $F_{x^{(j)}}^{l_j}$ has at most $s$ elements in their intersection. Thus, $\{F_{x^{(1)}}^{l_1}, \ldots, F_{x^(p)}^{l_p}\}$ is a monochromatic edge in $KG^{p}(n, k, s)$, which contradicts that $c$ is a proper coloring of $KG^{p}(n,k,s)$.    
\end{itemize}     
Now, by applying the $\mathbb{Z}_p$-Tucker lemma we have 
$$\underbrace{p(k-s-1)}_{\alpha}+ \underbrace{C}_{m-\alpha}(p-1)\geq n,$$ 
which implies that $C\geq\frac{n-p(k-s-1)}{p-1}$. So, $C\geq\left\lceil\frac{n-p(k-s-1)}{p-1}\right\rceil$ as $C$ is an integer. This is the desired conclusion.

The proof for the case that $r$ is not a prime number will be deduced from the next lemma. It is worth pointing out that the proof of the following lemma uses the same ideas as~\cite[Proposition 2.3]{Alon}.
\begin{lemma}
Let $r_1, r_2\geq 2$ be positive integers. Moreover, for $i=1, 2$, assume that 
$$\chi ({KG}^{r_i}(n,k,s))\geq\left\lceil\frac{n-r_i(k-s-1)}{r_i-1}\right\rceil,$$
for all integers $k, n, r_i, s$ satisfying $n\geq r_i(k-1)+1$, $k > s\geq 0$, and $r_i\geq 2$. Then, the following inequality
$$\chi ({KG}^{r_1r_2}(n,k,s))\geq\left\lceil\frac{n-r_1r_2(k-s-1)}{r_1r_2-1}\right\rceil,$$
is valid for all $n\geq r_1r_2(k-1)+1$, and $k> s\geq 0$.
\end{lemma}
\begin{proof}
Clearly the claim is true for $k=1$. Indeed, ${KG}^{r}(n,1,0)$ is the complete $r$-uniform hypergraph on $n$ vertices which its chromatic number is $\left\lceil\frac{n}{r-1}\right\rceil$. So, without loss of generality assume that $k\geq 2$. Put $r=r_1r_2$, and let $t=\left\lceil\frac{n-r(k-s-1)}{r-1}\right\rceil-1$. We need to show that $\chi ({KG}^{r}(n,k,s)) > t$. 
Suppose, contrary to our claim, that $\chi ({KG}^{r}(n,k,s))\leq t$. Assume $c : V\left({KG}^{r}(n,k,s)\right)\to\{1,\ldots, t\}$ be a proper $t$-coloring of ${KG}^{r}(n,k,s)$. Put 
$$m= (r_1-1)t+r_1(k-s-1)+1 = (r_1-1)(t-1)+ r_1(k-s)$$
We define a $t$-coloring of ${KG}^{r_2}(n,m,s)$ as follows. Take an $A\in V({KG}^{r_2}(n,m,s))$. Let ${KG}^{r_1}(A,k,s)$ be the induced sub-hypergraph of ${KG}^{r_1}(n,k,s)$ on the vertex set $\{X\subseteq A: |X|=k\}$. Note that, ${KG}^{r_1}(A,k,s)$ is isomorphic with ${KG}^{r_1}(m,k,s)$. The map $c$ gives naturally a $t$-coloring of ${KG}^{r_1}(A,k,s)$, i.e., $X\in V\left({KG}^{r_1}(A,k,s)\right)\mapsto c(X)$. But this coloring is not proper, as
\begin{align*}
m = & t(r_1-1)+ r_1(k-s-1)+1 \geq \\ & \left(\frac{n-r(k-s-1)}{r-1}-1\right)(r_1-1)+ r_1(k-s-1)+1  \geq\\ & \left(\frac{n-r_1(k-s-1)}{r_1-1}-1\right)(r_1-1)+ r_1(k-s-1)+1 =\\ & n-r_1+2\geq r_1r_2(k-1)+1-r_1+2 = r_1(r_2(k-1)-1)+3\underbrace{\geq}_{\text{as}\, r_1, r_2, k\geq 2}  r_1(k-1)+1,
\end{align*}
which implies $\chi({KG}^{r_1}(A,k,s)) > t$ by the induction hypothesis. So, there are $r_1$ distinct $k$-subsets $B_A^1, \ldots, B_A^{r_1}\subseteq A$ such that $c(B_A^1)=\cdots =c(B_A^{r_1})= c_{\star}$, and $|B_A^i\cap B_A^j|\leq s$ for any $i\neq j$. Now, define the color of $A$ as their common color, i.e., $c_{\star}$. Do the same procedure for all other vertices of ${KG}^{r_2}(n,m,s)$. This coloring is not proper as well, since
%$$\frac{n-r(k-s-1)}{r-1} > t \Longleftrightarrow n\geq\underbrace{(r-1)t+r(k-s-1)+1}_{(r-1)(t-1)+r(k-s)}.$$
\begin{align*}
  n\geq (r-1)(t-1)+r(k-s) & = (r_1r_2-r_2+r_2-1)(t-1)+ r_1r_2(k-s)\\ & =  (r_2-1)(t-1)+ r_2\underbrace{\left((r_1-1)(t-1)+ r_1(k-s)\right)}_{m},
\end{align*}
which implies $\chi({KG}^{r_2}(n,m,s)) > t$, by the induction hypothesis. Hence, there exist $A_1, \ldots , A_{r_2}\in V({KG}^{r_2}(n,m,s))$ such that $|A_i\cap A_j|\leq s$ and they have the same color. Thus,
$$E=\{B_{A_i}^{j} : 1\leq i\leq r_2, 1\leq j\leq r_1 \}$$ is a monochromatic edge in $KG^{r}(n,k,s)$. This contradicts that $c$ is a proper coloring of $KG^{r}(n,k,s)$ 
\end{proof}
\textbf{Remark:}
One can easily verify that ${KG}^{r}(n,k,0)$ is $\left\lceil\frac{n-r(k-1)}{r-1}\right\rceil$-colorable, see~\cite{Alon}. Hence, we have indeed an equality in $(1)$. So, it is of interest to know when the inequality presented in Theorem $1$ is sharp. Note that, if $k, r, s$ are fixed and $s\neq 0$, then $\chi\left({KG}^{r}(n,k,s)\right)$ can not be a linear function of $n$~\cite{Alon}. Therefore, in this case, for large enough $n$ we cannot have equality in Theorem 1. But, on the other hand, determining the chromatic number of $KG^{r}(n, k, s)$ whenever $k$ is around $\frac{n}{r}$ is also a challenging problem and studied with several authors, for this and related problems see~\cite{Balogh, Bobu, Bob}. Specially, in~\cite{Balogh} authors provided a proper coloring of $KG^{r}(n, k, s)$ with $O(s^2)$ colors where $k$ is around $\frac{n}{r}$ and $r$ is small comparing to $s$. Now, let us compare our lower bound with a lower bound that one could derive from the Alon--Frankl--Lov\'{a}sz result, for such parameters. Using theorem $1$, we get 
\begin{equation}
\chi ({KG}^{r}(rn,n,s))\geq\frac{r(s+1)}{r-1},
\end{equation}
while using the idea that there is a hypergraph homomorphism from the usual Kneser hypergraph $KG^r(rn-s,n-s, 0)$ to $KG^{r}(rn, n, s)$ and the Alon--Frankl--Lov\'{a}sz result, we could just get 
\begin{equation}
\chi ({KG}^{r}(rn,n,s))\geq\frac{r(s+1)}{r-1}-\frac{s}{r-1}.
\end{equation}
Indeed, the map which sends each $(n-s)$-subset $A\subseteq [rn-s]$ to $A\cup\{rn-s+1, \ldots, rn\}$ induces a hypergraph homomorphism from  $KG^r(rn-s,n-s, 0)$ to $KG^{r}(rn, n, s)$. 
Therefore, using the Alon--Frankl--Lov\'{a}sz result, we just get
$$\chi ({KG}^{r}(rn,n,s))\geq\chi ({KG}^{r}(rn-s,n-s,0))\geq\frac{rn-s-r(n-s-1)}{r-1}= \frac{r(s+1)}{r-1}-\frac{s}{r-1}.$$
In particular, inequality $(2)$ says $\chi\left(KG^{2}(2n, n, s)\right)\geq 2s+2$ while inequality $(3)$  just says $\chi\left(KG^{2}(2n, n, s)\right)\geq s+2$. Hence, our estimation for such a case is almost twice better! In conclusion, we can hope the presented lower bound in this paper will help in determining the chromatic number of generalized Kneser graphs for such parameters. 
%\section*{References}

\end{document}